\documentclass[english]{article}
\usepackage[T1]{fontenc}
\usepackage[latin9]{inputenc}
\usepackage{amsthm}
\usepackage{amsmath}
\usepackage{amssymb}

\makeatletter
\theoremstyle{plain}
\newtheorem{thm}{\protect\theoremname}[section]
\theoremstyle{plain}
\newtheorem{lem}[thm]{\protect\lemmaname}
\theoremstyle{plain}
\newtheorem{cor}[thm]{\protect\corollaryname}
\ifx\proof\undefined
\newenvironment{proof}[1][\protect\proofname]{\par
\normalfont\topsep6\p@\@plus6\p@\relax
\trivlist
\itemindent\parindent
\item[\hskip\labelsep
\scshape
#1]\ignorespaces
}{%
\endtrivlist\@endpefalse
}
\providecommand{\proofname}{Proof}
\fi
\newcommand{\lyxaddress}[1]{
\par {\raggedright #1
\vspace{1.4em}
\noindent\par}
}

\date{}

\makeatother

\usepackage{babel}
\providecommand{\corollaryname}{Corollary}
\providecommand{\lemmaname}{Lemma}
\providecommand{\theoremname}{Theorem}

\begin{document}

\title{Weak Gelfand Pair Property And Application To GL(n+1),GL(n) Over
Finite Fields}

\author{Yoav Ben Shalom}

\maketitle
Abstract. Let $F_{q}$ be the finite field with $q$ elements. Consider
the standard embedding $GL_{n}(F_{q})\hookrightarrow GL_{n+1}(F_{q})$.
In this paper we prove that for every irreducible representation $\pi$
of $GL_{n+1}(F_{q})$ over algebraically closed fields of characteristic
different from $2$ we have $dim\,\pi^{GL_{n}(F_{q})}\leq2$.

To do that we define a property of weak Gelfand pair and prove a generalization
of Gelfand trick for weak Gelfand pairs, using the anti-involution
transpose to get the result for $GL_{n+1}(F_{q})$,$GL_{n}(F_{q})$.
In a similar manner we show that for $q$ not a power of $2$ $(O_{n+1}(F_{q}),O_{n}(F_{q}))$
is a Gelfand pair over algebraically closed fields of characteristic
different from $2$.

\section{Introduction}

\subsection{Gelfand pairs}

A Gelfand pair is a pair of a group and its sub-group, usually denoted
by $G>H$, such that if $\pi$ is an irreducible representation of
$G$ then the dimension of the space of $H$ invariant vectors in
$\pi$ is at most 1. 

The theory was developed in the setting of Lie groups in a paper by
I. M. Gelfand {[}2{]}, and today has various applications. Finite
and infinite Gelfand pairs have been studied in asymptotic and geometric
group theory in connection with branch groups introduced by R.I. Grigorchuk.
Other applications were association schemes of coding theory, orthogonal
polynomials and special functions. Gelfand pairs are also used in
the study of finite Markov chains. For more uses and applications
see {[}1{]}.

\subsection{Structure of the proof}

The pair GL(n+1),GL(n) has been studied over various fields, and was
proven to be a Gelfand pair (better yet - a strong Gelfand pair) over
local fields (see {[}3{]}, {[}4{]}, {[}5{]}, {[}6{]}, {[}7{]}). Over
finite fields however it is not true, although we can still say something
about the pair. For some $q=p^{n}$ for $p$ prime, look at the groups
$GL_{n+1}(F_{q})\supset GL_{n}(F_{q})$. We want to be able to quantify
``how much'' they are not a Gelfand pair. One way to do that is
using the following lemma. Note that throughout the paper we will
only consider finite groups.
\begin{lem}
Gelfand trick. Let $H<G$ be groups. Suppose we have an anti-involution
$\sigma:G\rightarrow G$ that preserves all $H$ double-cosets. Then
over algebraically closed fields of characteristic zero $(G,H)$ is
a Gelfand pair.
\end{lem}
Looking at the lemma one might wonder what we can say if we have an
anti-involution that preserves almost all double-cosets - we still
want to have a property similar to being a Gelfand pair, but weaker.
We also want the result to be valid over more fields. We get such
a property using Gelfand-Kazhdan trick. Denoting by $F$ some algebraically
closed field of characteristic different from $2$, we acquire the
following lemma.
\begin{lem}
Let $H<G$ be groups, $\pi$ irreducible representation of $G$ over
$F$, $(\pi^{*})^{H}\neq0$, $\sigma$ anti-involution which preserves
$H$ and the central character of $\pi$, and the number of elements
of $H\setminus G/Z(G)H$ which are not preserved under $\sigma$ is
$2k$, then $dim\pi^{H}\leq k+1$.
\end{lem}
In some cases the assumption $(\pi^{*})^{H}\neq0$ can be removed
using the following lemma.
\begin{lem}
If the character of $F$ does not divide $|G|$, or if there exists
an anti-involution $\alpha$ of $G$ which preserves $H$ and conjugacy
classes, Then $dim(\pi^{H})=dim((\pi^{*})^{H})$.
\end{lem}
Finally, we use the last two lemmas for $GL_{n+1}(F_{q})$,$GL_{n}(F_{q})$
using the natural anti-involution transpose, and getting the $k$
from the following lemma.
\begin{lem}
For $n\geq1$, the number of $GL_{n}(F_{q})$ double-cosets in $GL_{n+1}(F_{q})$
(up to the center of $GL{}_{n+1}(F_{q})$) which are not preserved
under transpose is 2.\end{lem}
\begin{cor}
For every $\pi$ irreducible representation of $GL_{n+1}(F_{q})$
over $F$ we have $dim(\pi^{GL_{n}(F_{q})})\leq2$.
\end{cor}
For the case $F=\mathbb{C}$ it can be derived from {[}8{]}.

A similar result can be obtained for the orthogonal case.
\begin{thm}
For $q$ not a power of $2$, $(O_{n+1}(F_{q}),O_{n}(F_{q}))$ is
a Gelfand pair over $F$.
\end{thm}

\subsection{Acknowledgments}

The work was carried out under the guidance of Dmitry Gourevitch,
as part of the Kupcinet-Getz program at the Weizmann Institute of
Science. I owe a great deal of gratitude to Dmitry Gourevitch for
teaching me all I know about representation theory, and for his guidance
throughout this project.

\section{Weak Gelfand Pair Property}

We want to see some kind of ``weak'' Gelfand Pair property.
\begin{lem}
Let $G$ be a group, $\pi$ irreducible representation of $G$ over
$F$, $\sigma$ anti-involution, then for every $v_{1},v_{2}\in\pi,0\neq\varphi\in\pi^{*}$
if the matrix coefficients $\varphi(\pi(g)v_{1}),\varphi(\pi(g)v_{2})$
lie in $F(G)^{\sigma}$ then $v_{1},v_{2}$ are linearly dependent.\end{lem}
\begin{proof}
Define $B_{i}(\alpha_{1},\alpha_{2})=\varphi(\pi(\alpha_{1})\pi(\alpha_{2})v_{i})=\varphi(\pi(\alpha_{1}\star\alpha_{2})v_{i})=\varphi(\pi(\sigma(\alpha_{1}\star\alpha_{2}))v_{i})=\varphi(\pi(\sigma(\alpha_{2}))\pi(\sigma(\alpha_{1}))v_{i})$.
Since $\pi$ is irreducible, if $v_{i}\neq0$ then $\pi(\alpha_{2})v_{i}$
ranges over all $\pi$, and so the left kernels of $B_{i}$ are equal.
From the right hand side of the last equation, the right kernels are
equal too, i.e. $v_{1},v_{2}$ are in the kernels of the exact same
linear functionals and so $v_{i}$ are linearly dependent (since for
every two independent vectors there is a linear functional that sends
one to zero and the other not).\end{proof}
\begin{lem}
Let $H<G$ be groups, $\pi$ irreducible representation of $G$ over
$F$, $(\pi^{*})^{H}\neq0$, $\sigma$ anti-involution which preserves
$H$ and the central character of $\pi$, and the number of elements
of $H\setminus G/Z(G)H$ which are not preserved under $\sigma$ is
$2k$, then $dim\pi^{H}\leq k+1$.\end{lem}
\begin{proof}
Assume $dim\pi^{H}\geq k+2$. Note that by Schur's lemma the center
of $G$ acts by scalars on $\pi$. Let us denote the central character
by $Z_{\pi}$, and by $F(G)^{H\times H\times(Z(G),Z_{\pi})}$ the
space of $H\times H$ invariant functions in $F(G)$ such that multiplication
by an element of $Z(G)$ multiplies the result by the corresponding
scalar in $Z_{\pi}$. Using the assumption that the characteristic
of $F$ is not $2$, we can write this space in terms of even and
odd function: $(F(G)^{H\times H\times(Z(G),Z_{\pi})})^{\sigma}\oplus(F(G)^{H\times H\times(Z(G),Z_{\pi})})^{\sigma_{-}}$.
Note that the the space of odd functions is of dimension $k$. Now
take $0\neq\varphi\in(\pi^{*})^{H}$ and $k+2$ independent vectors
$v_{1},v_{2},\ldots,v_{k+2}\in\pi^{H}$. They create elements of $F(G)^{H\times H\times(Z(G),Z_{\pi})}$
by $\varphi(\pi(g)v_{i})$, and by lemma 2.1 they are linearly independent
in $F(G)^{H\times H\times(Z(G),Z_{\pi})}$, and so its dimension is
at least $k+2$. So we can take two linearly independent vectors $u_{1},u_{2}$
as a linear combination of the $v_{i}$ such that $\varphi(\pi(g)u{}_{i})$
are in $(F(G)^{H\times H\times(Z(G),Z_{\pi})})^{\sigma}$. But then
from the previous lemma we have linear dependency of $u_{1},u_{2}$
- contradiction.
\end{proof}
Finally, we want to show that the assumption $(\pi^{*})^{H}\neq0$
is not needed in some cases, with the following lemma.
\begin{lem}
If the character of $F$ does not divide $|G|$ then $dim(\pi^{H})=dim((\pi^{*})^{H})$.\end{lem}
\begin{proof}
Look at $\pi,\pi^{*}$ as representations of $H$. Notice that a preserved
vector under $H$ corresponds to the identity representation as a
sub-representation. In that sense, we get:
\[
dim(\pi^{H})=<\pi,1>,dim((\pi^{*})^{H})=<\pi^{*},1>
\]
and since $1=1^{*}$ and for every $\pi,\tau$ representations we
have $<\pi,\tau>=<\tau^{*},\pi^{*}>$ we get $dim(\pi^{H})=dim((\pi^{*})^{H})$
(using symmetry of intertwining number over $F$).
\end{proof}
This lemma shows us the assumption $(\pi^{*})^{H}\neq0$ is not needed
in specific cases, but we still want to give an easier constraint
than $(\pi^{*})^{H}\neq0$ for the general case. We can do that by
the following lemmas.
\begin{lem}
Let $\alpha\neq0$ be an anti-involution of $G$ which preserves conjugacy
classes. Define a representation of $G$ on the space of $\pi$ by
$\pi'(g)=\pi(\alpha(g))^{-1}$. Then $\pi^{*}\cong\pi'$.\end{lem}
\begin{proof}
We will show that the modular characters of $\pi^{*},\pi'$ are equal,
and hence (from a theorem by R. Brauer) they are isomorphic. $\pi^{*}$
can be defined as $(\pi(g)^{t})^{-1}$ acting on the space of $\pi$.
Hence we need to show that the modular characters of $\pi(g)^{t},\pi(\alpha(g))$
are equal, which means that they have the same eigenvalues. Since
$\alpha$ preserves conjugacy classes, there is some $a\in G$ such
that $\alpha(g)=aga^{-1}$, which gives us
\[
\pi(\alpha(g))=\pi(aga^{-1})=\pi(a)\pi(g)\pi(a^{-1})=\pi(a)\pi(g)\pi(a)^{-1}
\]

which means that $\pi(g),\pi(\alpha(g))$ are conjugate, and hence
have the same eigenvalues. Finally, we know transpose preserves eigenvalues
and we get $\pi^{*}\cong\pi'$.
\end{proof}
If $\alpha$ also preserves $H$ we get the wanted result.
\begin{lem}
Let $\alpha$ be an anti-involution of $G$ which preserves $H$ and
conjugacy classes. Then $dim(\pi^{H})=dim((\pi^{*})^{H})$.\end{lem}
\begin{proof}
From lemma 3.1 it's enough to show that $dim(\pi^{H})=dim((\pi')^{H})$
which is obvious because for a vector to be preserved under $H,$
it needs to be preserved under
\[
\pi(\alpha(H))^{-1}=\pi(H)^{-1}=\pi(H^{-1})=\pi(H)
\]

\end{proof}
Notice that if we take $\alpha=\sigma$ then the only assumption we
need to add to lemma 2.2 instead of $(\pi^{*})^{H}\neq0$ is that
$\sigma$ preserves conjugacy classes.

\section{Application To GL(n+1),GL(n) Over Finite Fields}

To use lemma 2.2 for $GL_{n+1}(F_{q})$,$GL_{n}(F_{q})$ we need to
calculate $k$.
\begin{lem}
For $n\geq1$, the number of $GL_{n}(F_{q})$ double-cosets in $GL_{n+1}(F_{q})$
(up to the center of $GL{}_{n+1}(F_{q})$) which are not preserved
under transpose is 2 (i.e. $k=1$).\end{lem}
\begin{proof}
A general double coset is of the form 
\[
\begin{pmatrix}B & 0\\
0 & 1
\end{pmatrix}A^{'}\begin{pmatrix}C & 0\\
0 & 1
\end{pmatrix}
\]

where $B\in GL_{n}(F_{q})$, $C\in GL_{n}(F_{q})$ and $A^{'}\in GL_{n+1}(F_{q})$.
We will write $A^{'}=\begin{pmatrix}A & v\\
\varphi & \lambda
\end{pmatrix}$, where $A\in M_{n}(F_{q})$, $v,\varphi$ are vectors in $F_{q}^{n}$
and $\lambda$ is a scalar. Because $A^{'}$ is in $GL_{n+1}(F_{q})$,
it is easy to see that $rank(A)$ is either $n$ or $n-1$. Let us
assume $rank(A)=n$ (alternatively $n-1)$. In such a case, we know
that using row and column operation we can bring $A$ to the form
$I_{n}$ ($I_{n-1})$. Note that we can apply such operations using
$B$ and $C$. therefore, we can look only at the case where $A=I_{n}$
($I_{n-1}$). Note that then, after matrix multiplication, the upper
left $n\times n$ matrix is $BC$ ($BI_{n-1}C$). Because it needs
to be of the form of $A^{t}$, we get $BC=I_{n}$ ($BI_{n-1}C=I_{n-1}$),
i.e. $C=B^{-1}$. Finally let us write down the multiplication explicitly:
\[
\begin{pmatrix}B & 0\\
0 & 1
\end{pmatrix}\begin{pmatrix}I_{n} & v\\
\varphi & \lambda
\end{pmatrix}\begin{pmatrix}B^{-1} & 0\\
0 & 1
\end{pmatrix}=\begin{pmatrix}B & Bv\\
\varphi & \lambda
\end{pmatrix}\begin{pmatrix}B^{-1} & 0\\
0 & 1
\end{pmatrix}=\begin{pmatrix}I_{n} & Bv\\
\varphi B^{-1} & \lambda
\end{pmatrix}
\]
\[
\left(alternatively\,\begin{pmatrix}I_{n-1} & Bv\\
\varphi C^{-1} & \lambda
\end{pmatrix}\right)
\]

We need to verify that there is such $B$ so that it equals $A^{t}$.
So we get the following equalities:
\[
Bv=\varphi^{t}\,\,\,\,\varphi B^{-1}=v^{t}
\]

multiply the left one by $B$ and take transpose, we get $B^{t}v=\varphi^{t}$.
First assume $\varphi,v\neq0$, and so it's enough to find a symmetric
matrix $B\in GL_{n}(F_{q})$ such that $Bv=\varphi^{t}$. This will
be proved in the next lemma. So now assume $\varphi$ or $v$ is $0$
(Therefore the case of rank $n-1$ is dismissed, because then we'll
have a row/column of zeros). It's easy to see that in this case there
is no sufficient matrix $B$, and so we are left with checking how
many such double-cosets exist. Assuming $v=0$ and $\varphi_{1}\neq0$
from the above calculation we see that all such matrix are in the
same double-coset, because it is enough to find an invertible matrix
$B$ such that $\varphi_{1}B^{-1}=\varphi_{2}$ which is obvious.
So up to the center of $GL{}_{n+1}(F_{q})$ (i.e. scalar multiplication,
which allows us to assume $\lambda=1,0$, but since $\lambda=0$ means
$\varphi\neq0,v\neq0$ we may assume $\lambda=1$) we get 2 double-cosets
($\varphi=0,\, v\neq0$, $\varphi\neq0,\, v=0$), which are transpose
of each other.\end{proof}
\begin{lem}
Over $F_{q}$, for every n-vectors $\varphi,v\neq0$ there exists
a symmetric matrix $B\in GL_{n}(F_{q})$ such that $B\varphi=v$.\end{lem}
\begin{proof}
We use induction, Using the following form:
\[
\begin{pmatrix}a_{1}\\
v_{1}
\end{pmatrix}=\begin{pmatrix}c & r_{1}^{t}\\
r_{1} & A
\end{pmatrix}\begin{pmatrix}b_{1}\\
\varphi_{1}
\end{pmatrix}
\]

Were $a_{1},b_{1},c$ are scalars, $v_{1},r_{1},\varphi_{1}$ are
(n-1)-vectors and $A$ is an (n-1)-square matrix. Let's first solve
the trivial case - if $a_{1},b_{1},v_{1},\varphi_{1}\neq0$ then set
$c=a_{1}/b_{1}$, $r_{1}=0$ and acquire $A$ by using induction on
$v_{1}=A\varphi_{1}$. If $\varphi_{1}=0$, then since $\varphi\neq0$
we get $b_{1}\neq0$. So obviously $c=a_{1}/b_{1},r_{1}=\frac{1}{b_{1}}v_{1}$.
All we need now is a suitable $A$, with no restraints other than
to be symmetric, and so that the whole matrix $B$ will be invertible.
By row and column swaps we can assume that the non-zero elements of
$r_{1}$ are the upper-most. If $r_{1}=0$ then $a_{1}\neq0$ and
we can take $A=I_{n-1}$. So assume $r_{1}\neq0$ and we choose the
following $A$ (we show here the whole $B$):
\[
\begin{pmatrix}\frac{a_{1}}{b_{1}} & [r_{1}]_{1} & [r_{1}]_{2} & [r_{1}]_{3} & \ldots & [r_{1}]_{p} & 0 & \ldots & 0\\
{}[r_{1}]_{1} & 0 & 0 & 0 & 0 & 0\\
{}[r_{1}]_{2} & 0 & 1 & 0 & 0 & 0\\
{}[r_{1}]_{3} & 0 & 0 & 1 & 0 & 0\\
\vdots & 0 & 0 & 0 & \ddots & 0\\
{}[r_{1}]_{p} & 0 & 0 & 0 & 0 & 1\\
0 &  &  &  &  &  & 1 & 0 & 0\\
\vdots &  &  &  &  &  & 0 & \ddots & 0\\
0 &  &  &  &  &  & 0 & 0 & 1
\end{pmatrix}
\]
By row and column operations, and by the assumption $[r_{1}]_{1}\neq0$
we can bring it to the form:
\[
\begin{pmatrix}0 & 1 & 0 & 0 & \ldots & 0 & 0 & \ldots & 0\\
1 & 0 & 0 & 0 & 0 & 0\\
0 & 0 & 1 & 0 & 0 & 0\\
0 & 0 & 0 & 1 & 0 & 0\\
\vdots & 0 & 0 & 0 & \ddots & 0\\
0 & 0 & 0 & 0 & 0 & 1\\
0 &  &  &  &  &  & 1 & 0 & 0\\
\vdots &  &  &  &  &  & 0 & \ddots & 0\\
0 &  &  &  &  &  & 0 & 0 & 1
\end{pmatrix}
\]

Which is invertible, and hence $B$ is invertible. Since $B^{-1}$
is symmetric iff $B$ is symmetric, the case $\varphi_{1}=0$ solves
the case $v_{1}=0$, and all is left is the case $a_{1}=0$ or $b_{1}=0$,
which are again the same. So assume $b_{1}=0$. By induction there's
a symmetric invertible $A$ such that $v_{1}=A\varphi_{1}$. Now we
need to find suitable $r_{1},c$. We choose arbitrary $r_{1}$ - with
the only constrain that $<r_{1},\varphi_{1}>=a_{1}$, which is easy.
for $c$ we have two options: 0 or 1. If both don't work, it means
that the rows $2$ to $n$ have the vector $e_{1}$ in their span.
But if we look at the vectors from their second place, they are vectors
in $A$ which are linearly independent, and so all the coefficients
are 0, and we get a contradiction. Hence 0 or 1 is always good, and
that concludes the proof. 
\end{proof}
Combining lemmas 2.2, 2.5 and 3.1 (for the case $\sigma=\alpha=transpose$)
we finally get the result for $G=GL_{n+1}(F_{q})$, $H=GL_{n}(F_{q})$.
\begin{cor}
For every $\pi$ irreducible representation of $GL_{n+1}(F_{q})$
over $F$ we have $dim(\pi^{GL_{n}(F_{q})})\leq2$.
\end{cor}

\section{O(n+1),O(n) Over Finite Fields}
\begin{thm}
For $q$ not a power of $2$, $(O_{n+1}(F_{q}),O_{n}(F_{q}))$ is
a Gelfand pair over $F$.\end{thm}
\begin{proof}
We use lemmas 2.2, 2.5 with $\alpha=\sigma=transpose$. Denote $G=O(n+1),H=O(n)$.
Since $H\backslash G/H\cong G\backslash(G/H\times G/H)$ (which is
true for every groups $G,H$), and since $G/H$ is isomorphic to the
unit sphere, it's enough to show that for any unit vectors $u,v$
there is an element $g\in G$ such that $gu=v,gv=u$ (since transpose
translates to $(u,v)\mapsto(v,u)$). If $u-v$ is not orthogonal to
itself, take $g(x)=x-2\frac{<u-v,x>}{<u-v,u-v>}(u-v)$ (i.e. the reflection
relative to the hyperplane orthogonal to $u-v)$.
\begin{eqnarray*}
g(u) & = & u-2\frac{<u-v,u>}{<u-v,u-v>}(u-v)=u-2\frac{1-<u,v>}{2-2<u,v>}(u-v)=v\\
g(v) & = & v-2\frac{<u-v,v>}{<u-v,u-v>}(u-v)=v-2\frac{<u,v>-1}{2-2<u,v>}(u-v)=u
\end{eqnarray*}
If $u-v$ is orthogonal to itself, notice that then 
\[
0=<u-v,u-v>=2-2<u,v>\Rightarrow<u,v>=1
\]

and therefore
\[
<u+v,u+v>=2+2<u,v>=4\neq0
\]

so we can take $g(x)=2\frac{<u+v,x>}{<u+v,u+v>}(u+v)-x=\frac{<u+v,x>}{2}(u+v)-x$
(i.e. the reflection relative to $u+v$).
\begin{eqnarray*}
g(u) & = & \frac{<u+v,u>}{2}(u+v)-u=(u+v)-u=v\\
g(v) & = & \frac{<u+v,v>}{2}(u+v)-v=(u+v)-v=u
\end{eqnarray*}
\end{proof}

\lyxaddress{Y. Ben Shalom, The School of Mathematical Sciences, Tel Aviv University,
Israel.}

\lyxaddress{E-mail address: yoavben@tau.ac.il}
\end{document}